\documentclass[11pt]{amsart}
\usepackage{amssymb, amscd}
\usepackage{latexsym,epsfig}
\usepackage[all]{xy}
\usepackage{pst-all}
\usepackage{pstcol}
\usepackage[normalem]{ulem}

\usepackage[
        hypertexnames=false,
        hyperindex,
        pagebackref,
        pdftex,
        breaklinks=true,
        bookmarks=false,
        colorlinks,
        linkcolor=blue,
        citecolor=red,
        urlcolor=red,
]{hyperref}

\usepackage[mathscr]{eucal}

\def\vandaag{\number\day\space\ifcase\month\or
 januari\or februari\or  maart\or  april\or mei\or juni\or  juli\or
 augustus\or  september\or  oktober\or november\or  december\or\fi,
\number\year}
\def\today{\ifcase\month\or
 Jan\or Febr\or  Mar\or  Apr\or May\or Jun\or  Jul\or
 Aug\or  Sep\or  Oct\or Nov\or  Dec\or\fi
 \space\number\day, \number\year}

\begin{document}

%%%%%%%%%%%%%%%%%%%%%%%%%%%%%%%%%%%%%%%%%%%%%%%%%%%%%%%%%%%%%%
\newtheorem{theorem}{Theorem}[section]
\newtheorem{lemma}[theorem]{Lemma}
\newtheorem{proposition}[theorem]{Proposition}
\newtheorem{corollary}[theorem]{Corollary}
\newtheorem{conjecture}[theorem]{Conjecture}
\newtheorem{definition-lemma}[theorem]{Definition-Lemma}
\newtheorem{claim}[theorem]{Claim}
\theoremstyle{definition}
\newtheorem{definition}[theorem]{\bf Definition}
\newtheorem{example}[theorem]{\bf Example}
\newtheorem{none}[theorem]{}
\theoremstyle{remark}
\newtheorem{remark}[theorem]{\bf Remark}
\newtheorem{conclusion}[theorem]{\bf Conclusion}
\newtheorem{remarks}[theorem]{\bf Remarks}
\newtheorem{question}[theorem]{\bf Question}
\newtheorem{problem}[theorem]{\bf Problem}
%%%%%%%%%%%%%%%%%%%%%%%%%%%%%%%%%%%%%%%%%%%%%%%%%%%%%%%%%%%%%%%%%%%
\newcommand{\CC}{\mathbb C}
\newcommand{\DD}{\mathbb D}
\newcommand{\EE}{\mathbb E}
\newcommand{\FF}{\mathbb F}
\newcommand{\GG}{\mathbb G}
\newcommand{\HH}{\mathbb H}
\newcommand{\LL}{\mathbb L}
\newcommand{\PP}{\mathbb P}
\newcommand{\QQ}{\mathbb Q}
\newcommand{\RR}{\mathbb R}
\newcommand{\VV}{\mathbb V}
\newcommand{\WW}{\mathbb W}
\newcommand{\ZZ}{\mathbb Z}
%%%%%%%%%%%%%%%%%%%%%%%%%%%%%%%%%%%%%%%%%%%%%%%%%%%%%%%%%%%%%5
\newcommand{\A}{\mathcal A}
\newcommand{\R}{\mathcal R}
\newcommand{\cO}{\mathcal O}
\newcommand{\cX}{\mathcal X}

\newcommand{\cI}{\mathcal I}
\newcommand{\sA}{{\mathcal A}^{\ast}}
\newcommand{\bM}{\overline{\mathcal M}}
\newcommand{\bA}{\overline{\mathcal A}}
\newcommand{\sF}{\mathscr {F}}
\newcommand{\cH}{\mathcal {H}}
\newcommand{\sH}{\mathscr {H}}
\newcommand{\sL}{\mathscr {L}}
\newcommand{\sZ}{\mathscr {Z}}

%%%%%%%%%%%%%%%%%%%%%%%%%%%%%%%
\newcommand\Sp{\operatorname{Sp}}
\newcommand\ch{\operatorname{ch}}
\newcommand\Fr{\operatorname{Fr}}
\newcommand\spec{\operatorname{Spec}}
%%%%%%%%%%%%%%%%%%%%%%%%%%%%%%%
\newcommand\an{\mathrm{an}}
\newcommand\gp{\mathrm{gp}}
%%%%%%%%%%%%%%%%%%%%%%%%%%%%%%%%%%%
\title[]{Modular forms of degree two \\
and curves of genus two\\ 
in characteristic two}

\author{Fabien Cl\'ery}
\address{Department of Mathematics,
Loughborough University,
England, and Institute of Computational and Experimental Research in Mathematics, 
121 South Main Street, Providence, RI 02903, USA
}
\email{cleryfabien@gmail.com}

\author{Gerard van der Geer}
\address{Korteweg-de Vries Instituut, Universiteit van Amsterdam, Science Park 904, 1098 XH Amsterdam, The Netherlands, and 
Universit\'e du Luxembourg, Unit\'e de Recherche en Math\'ematiques,
L-4364 Esch-sur-Alzette, Luxembourg.}
\email{g.b.m.vandergeer@uva.nl}

\subjclass{14F46,11F70,14J15}

\maketitle
%\centerline{\tt \today}
%%%%%%%%%%%%%%%%%%%%%%%%%%%%%%%%%%%%%%%%
\begin{abstract}
We describe the ring of modular forms of degree $2$ in characteristic
$2$ using its relation with curves of genus $2$.
\end{abstract}

%%%%%%%%%%%%%%%%%%%%%%%%%%%%%%%%%%%%%%%%

\begin{section}{Introduction}
In the 1960s Igusa determined the ring of Siegel modular forms of degree~$2$
using the close connection between
the moduli of principally polarized abelian surfaces and the moduli of curves of 
genus $2$, see \cite{Igusa1960}. 
Since curves of genus $2$ in characteristic different from $2$ can be described by
binary sextics he could use the invariant theory of binary sextics and established
the link with modular forms using Thomae's formulas that express theta constants
in terms of cross ratios of zeros of binary sextics. Recently, in joint work with Carel Faber we 
exhibited in \cite{CFG-MathAnn}
a more direct version of this
without the detour of Thomae's formulas  and used it for describing vector-valued
Siegel modular forms of degree $2$. 

In positive characteristic modular forms of degree $2$ are described as sections of
powers of the determinant bundle $L$ of the Hodge bundle ${\EE}$ on the moduli space
$\mathcal{A}_2 \otimes {\FF}_p$ of principally polarized abelian surfaces in
characteristic $p$. It turns out that for $p\geq 5$ the situation is very close
to the case of characteristic $0$. For $p\geq 5$ the graded ring 
$$
\mathcal{R}_2({\FF}_p)=\oplus_k H^0(\mathcal{A}_2 \otimes {\FF}_p, L^k)
$$
has generators of weight $4,6,10,12$ and $35$ and there is a relation of degree
$70$ expressing the square of the odd weight generator in terms of the even weight
generators, just as in characteristic $0$. In fact, Igusa determined the ring 
$\mathcal{R}_2({\ZZ})$ of integral modular forms and the reduction map 
$\mathcal{R}_2({\ZZ})\to \mathcal{R}_2({\FF}_p)$ is surjective for $p\geq 5$. 
We refer to
\cite{Igusa1979, Ichikawa, B-N,N,N2}.

The situation in characteristic $2$ and $3$ is different. The ring 
$\mathcal{R}_2({\FF}_3)$ was determined in \cite{vdG-char3} using the relation with 
binary sextics. 

In characteristic $2$ curves of genus $2$ can no longer be described by binary sextics.
Nevertheless, as Igusa showed in \cite{Igusa1960} there is still a close relationship
with the invariant theory of binary sextics. In this paper we use invariants of
curves of genus $2$ in characteristic $2$ to determine the ring of modular forms
of degree $2$. 
The result is as follows.

\begin{theorem}
The ring $\mathcal{R}_2({\FF}_2)$ is generated by modular forms of weights 
$1$, $10$, $12$, $13$, $48$
with one relation of weight $52$:
$$
\mathcal{R}_2({\FF}_2) = {\FF}_2[\psi_1,\chi_{10},\psi_{12},\chi_{13}, \chi_{48}]/(R)
$$
with $R=\chi_{13}^4 + \psi_1^3\chi_{10}\chi_{13}^3+
\psi_1^4 \chi_{48} +
\chi_{10}^4 \psi_{12} $.
The ideal of cusp forms is generated by $\chi_{10},\chi_{13}$ and $\chi_{48}$.
\end{theorem}

The modular form $\psi_1$ is the Hasse invariant vanishing on the locus
of non-ordinary abelian surfaces; the form $\chi_{10}$ 
vanishes doubly on the
locus of abelian surfaces that are products of elliptic curves. All the forms
$\psi_1, \chi_{10}, 
\psi_{12}, \chi_{13}$ and $\chi_{48}$ are constructed using invariant
theory, see Section~\ref{InvariantsandMF}.
 
We point out that we find non-zero 
(regular) vector-valued modular forms
of weights not allowed in characteristic $0$, like $(3,-1)$ or $(2,0)$.

%\section*{Acknowledgements}
We thank the referees for helpful remarks.

\end{section}
\begin{section}{Modular Forms}
We denote by $\mathcal{A}_g$ the moduli stack of principally polarized abelian varieties
of dimension $g$ and by ${\EE}_g$ the Hodge bundle on $\mathcal{A}_g$. It extends to
a Faltings-Chai type toroidal compactification $\tilde{\mathcal{A}}_g$. 
We write 
$L=\det({\EE}_g)$. For $g>1$ sections of $L^k$ over $\mathcal{A}_g$ 
extend automatically to $\tilde{\mathcal{A}}_g$, a property known as 
the Koecher principle, see \cite[Prop. 1.5 page 140]{F-C}.
We write $M_k(\Gamma_g)=H^0(\tilde{\mathcal{A}}_g\otimes {\FF}_2, L^k)$ for the space of
scalar-valued modular forms of weight $k$. Moreover, we set
$$
\mathcal{R}_g({\FF}_2)= \oplus_k M_k(\Gamma_g)\, ,
$$ 
the ring of scalar-valued modular forms of degree $g$ in characteristic $2$.
It is a finitely generated ${\FF}_2$-algebra. 
It is well-known that $H^0(\tilde{\mathcal{A}}_g, L^k)=(0)$ for $k<0$.
For $g=1$ we know by Deligne \cite{Deligne1975} that
$\mathcal{R}_1({\FF}_2)={\FF}_2[c_1,\Delta]$ with $c_1$ of weight $1$ and 
$\Delta$ a cusp form of weight $12$.

We are interested in the case $g=2$. In the following we shall simply write
${\EE}$ for the Hodge bundle ${\EE}_2$.

In the Chow group with rational coefficients 
of codimension~$1$ cycles of $\tilde{\A}_2$  the class of the locus
$\mathcal{A}_{1,1}$ of products of elliptic curves satisfies the relation
$$
2[\overline{\A}_{1,1}]+[D]= 10\, \lambda_1 \, ,
\eqno(1)
$$
where $\lambda_1$ is the first Chern class of ${\EE}$ and $D$ 
the divisor added to compactify $\mathcal{A}_2$, see \cite[(8.4)]{Mumford}. Therefore 
there exists a non-zero modular form $\chi_{10}$ of weight $10$ over ${\FF}_2$ 
vanishing doubly on the divisor  $\overline{\A}_{1,1} \otimes {\FF}_2$.
As the cycle relation shows $\chi_{10}$ is a cusp form. The form is determined up
to a non-zero multiplicative constant and we shall normalize $\chi_{10}$ later.

Since $\mathcal{A}_{1,1}$ is the image of a degree $2$ morphism
$\mathcal{A}_1 \times \mathcal{A}_1 \to \mathcal{A}_{1,1}$ 
by restriction to the locus $\mathcal{A}_{1,1}\otimes{\FF}_2$ 
we find 
an exact sequence
$$
0 \to M_{k-10}(\Gamma_2) \to M_k(\Gamma_2) \to {\rm Sym}^2(M_k(\Gamma_1))\, .
\eqno(2)
$$

Another divisor that we will use is the locus $V_1$ 
of principally polarized abelian surfaces
of $2$-rank $\leq 1$. The cycle class of this locus is known; it is the vanishing locus
of the map $L\to L^{\otimes 2}$ induced by Verschiebung on the universal abelian surface
and we thus find (cf.\ \cite{vdG})
$$
[V_1] = \lambda_1 \, .
$$
Hence there exists a modular form $\psi_1$ of weight $1$, determined up to
a non-zero multiplicative constant,
with divisor $V_1$. It is called the Hasse invariant. This is not a cusp form
since it restricts to the Hasse invariant $c_1$ on the boundary component 
$\mathcal{A}_1^{\ast}\otimes{\FF}_2$ in the Satake
compactification $\mathcal{A}_2^{\ast} \otimes {\FF}_2$.  
We shall normalize $\psi_1$ later.

The existence of $\psi_1$ 
implies $\dim M_k(\Gamma_2) \leq \dim M_{k+1}(\Gamma_2)$. 
The exact sequence (2) and the structure of $\mathcal{R}_1({\FF}_2)$ imply that
$$
\dim M_k(\Gamma_2)=1 \quad \text{\rm for $k=1,\ldots,9$ \quad and $\dim M_{10}(\Gamma_2)=2$. } \eqno(3)
$$
In order to estimate $\dim M_k(\Gamma_2)$ for larger $k$  
we can apply semi-continuity (cf.\ \cite[Thm.\ 12.8]{Hartshorne} )
to $L$ defined over $\tilde{\mathcal{A}}_2 / {\ZZ}$
and deduce $\dim H^0(\tilde{\mathcal{A}}_2 \otimes {\CC},L^k) \leq 
\dim H^0(\tilde{\mathcal{A}}_2 \otimes {\FF}_2,L^k)$.
Using this  and the exact sequence (2) 
we can deduce
$$
\dim M_{11}(\Gamma_2)=2, \quad \text{\rm and $3\leq \dim M_{k}(\Gamma_2) \leq 4$ for $ 12\leq k \leq 19$, }  \eqno(4)
$$
but we will deduce this in another way later, see (7).

Similarly, we can consider vector-valued modular forms. 
For each irreducible representation
$\rho$ of ${\rm GL}(2,{\QQ})$ 
there is a corresponding bundle ${\EE}_{\rho}$. Since we are dealing
with $g=2$ the $\rho$ in question correspond to pairs $(j,k)$ of integers with $j\geq 0$
and ${\EE}_{\rho}={\rm Sym}^j({\EE}) \otimes \det({\EE})^{\otimes k}$.
We write $M_{j,k}(\Gamma_2)$ for $H^0(\mathcal{A}_{2}\otimes {\FF}_2,
{\rm Sym}^j({\EE})\otimes \det({\EE})^{\otimes k})$ and $S_{j,k}(\Gamma_2)$ for the space
of cusp forms. Again, these sections extend over $\tilde{\mathcal{A}}_2 \otimes {\FF}_2$ by
the Koecher principle. But note that $\rho \otimes {\FF}_2$ might be reducible.

\begin{lemma} \label{dim68}
We have $\dim S_{6,8}(\Gamma_2)\geq 1$ and every element of $S_{6,8}(\Gamma_2)$ vanishes on $\mathcal{A}_{1,1}\otimes {\FF}_2$. 
If $\dim S_{6,8}(\Gamma_2)\geq 2$ then $M_{6,-2}(\Gamma_2)\neq 0$.
\end{lemma}
\begin{proof}
  We know that over ${\CC}$ we have
$\dim S_{6,8}(\Gamma_2)=1$. Hence by semi-continuity of dimensions 
there exists a cusp form of weight $(6,8)$ in characteristic $2$.
Note that the Hodge bundle ${\EE}$ pulled back via
$$
\mathcal{A}_1 \times {\mathcal A}_1 \to \mathcal{A}_{1,1}\hookrightarrow \mathcal{A}_2
$$
can be written as $p_1^*({\EE}_1) \oplus p_2^*({\EE}_1)$ 
with ${\EE}_1$ the Hodge bundle of $\mathcal{A}_1$ and $p_1,p_2$ the two projections.
The pullback of $\chi_{6,8}$ to 
$(\mathcal{A}_1 \times \mathcal{A}_1)\otimes {\FF}_2$ must vanish because it lands in
$$
\bigoplus_{j=0}^6 S_{14-j}(\Gamma_1)\otimes S_{8+j}(\Gamma_1)
$$
and this is $(0)$. If we develop a form $\varphi \in S_{6,8}(\Gamma_2)$ 
vanishing on $\mathcal{A}_{1,1} \otimes {\FF}_2$ 
in the normal direction of the divisor
$\mathcal{A}_{1,1}\otimes {\FF}_2$ we see that the first (`non-constant') 
term in its Taylor expansion lands in 
$$
\bigoplus_{j=0}^6 S_{15-j}(\Gamma_1)\otimes S_{9+j}(\Gamma_1)
$$
and the only non-zero factor here is $S_{12}(\Gamma_1)\otimes S_{12}(\Gamma_1)$
and the image 
 is symmetric under the interchange of factors. But $\dim S_{12}(\Gamma_1)=1$.
Therefore, if $\varphi_1, \varphi_2$ are two linearly independent cusp forms
of weight $(6,8)$ there exists a non-trivial linear combination that
vanishes with order $\geq 2$ along $\mathcal{A}_{1,1} \otimes {\FF}_2$. 
This yields a modular form divisible by $\chi_{10}$, hence a regular 
section of ${\rm Sym}^6({\EE}) \otimes \det({\EE})^{-2}$, not identically zero on
$\mathcal{A}_{1,1}\otimes {\FF}_2$. 
\end{proof}

We thus have at least one non-zero modular form $f \in S_{6,8}(\Gamma_2)$ 
that vanishes with multiplicity $1$ on
$\mathcal{A}_{1,1}\otimes {\FF}_2$. Thus we also find a rational modular form 
$f/\chi_{10}$ with possible 
poles along $\mathcal{A}_{1,1} \otimes {\FF}_2$.

We can analyze the order of the poles of 
the coordinates of $f/\chi_{10}$ near a generic
point of $\mathcal{A}_{1,1}$. We consider $f$
near a generic point of $\mathcal{A}_{1,1}\otimes {\FF}_2$ and write its
Taylor expansion  as
$$
f= (\gamma_0,\ldots,\gamma_6) \quad
\text{\rm with $\gamma_i=\sum_{r\geq 0} \gamma_{i,r} t^r$}
$$
where $t$ is a local normal coordinate.
If for fixed $i$ we have $\gamma_{i,r}=0$ for $r<r_0$ then 
$$
\gamma_{i,r_0} \in QS_{14+r_0-i}(\Gamma_1)\otimes QS_{8+r_0+i}(\Gamma_1)
$$
where $QS$ denotes quasi-modular cusp forms, see \cite[Sec.\ 5]{Zagier}
for the notion of quasi-modular form.
From the dimensions of $S_k(\Gamma_1)$ we see as in the proof of Lemma
\ref{dim68}
that ${\rm ord}_{\mathcal{A}_{1,1}}(\gamma_3)\geq -1$ and we can 
even get the estimate
$$
{\rm ord}_{\mathcal{A}_{1,1}}(\gamma_0,\ldots,\gamma_6) 
\geq (2,1,0,-1,0,1,2)\, ,
$$
but we will not use this.

\end{section}
%s%%%%%%%%%%%%%%%%%%%%%%%%%%%%%%%%%%%%%f
\begin{section}{Curves of genus $2$ in characteristic $2$}
Let $C$ be a smooth projective curve of genus $2$
over a perfect field $k$ of characteristic $2$. 
Let $K$ be the canonical divisor (class) of $C$. Then by Riemann-Roch
$\dim H^0(C,O(nK))=2n-1$ for $n\geq 2$.
If $\xi_0,\xi_1$ is a basis of $H^0(C,O(K))$, then there is an element $\eta
\in H^0(C,O(3K))$ such that $\xi_0^3,\xi_0^2\xi_1,\xi_0\xi_1^2,\xi_1^3, \eta$
form a basis of $H^0(C,O(3K))$. Then 
$$
\xi_0^6,\xi_0^5\xi_1,\ldots,\xi_1^6, \, \xi_0^3\eta,\xi_0^2\xi_1\eta,\xi_0\xi_1^2\eta,\xi_1^3\eta
$$
are linearly independent in $H^0(C,O(6K))$ and since $\eta^2 \in H^0(C,O(6K))$ there
must be a linear relation
$$
\eta^2+ c_3\eta + c_6 =0
$$ 
with $c_3$ (resp.\ $c_6$) homogeneous of degree $3$
(resp.\ $6$) in $\xi_0,\xi_1$.
Setting $x=\xi_1/\xi_0$ and $y=\eta/\xi_0^3$ we can write the equation as
$$
y^2+a \, y =b \eqno(5)
$$
with $a \in k[x]$ non-zero of degree $\leq 3$ and $b \in k[x]$ of degree $\leq 6$. 
The hyperelliptic involution of the curve $C$ is given by $y \mapsto y+a$.
The fixed points are given by the equation $a=0$ 
(or $c_3=0$ in projective coordinates); 
hence there are $3$, $2$ or $1$ ramification points. This corresponds to the $2$-rank
of ${\rm Jac}(C)$ being equal to $2$, $1$ or $0$.

We point out that a curve of genus $2$ defined by an equation (5)  
comes with a basis ${d}x/a$, $x\, {d}x/a$ of $H^0(C,O(K))$.

The {\sl choices} we made for arriving at equation (5) 
are a basis of $H^0(C,O(K))$ and an element of $\eta$ in $H^0(C,O(3K))$.
Clearly, another choice of $\eta$ is of the form $\epsilon \eta + \theta $ with $\epsilon$ a unit in $k$
and $\theta$ homogeneous of degree $3$ in $\xi_0,\xi_1$, or $y\to uy+v$
for $u\in k^*$ and $v\in k[x]$ of degree $\leq 3$.

The induced action on
the pair $(a,b)$ is by 
$$
(a,b) \mapsto (a/u,(b+v^2+av)/u^2) \quad
\text{\rm with $u\in k^*$, $v\in k[x]$ of degree $\leq 3$}\, .
$$
Another choice of basis of $H^0(C,O(K))$ is given by an element of
${\rm GL}(V)$ with $V=\langle \xi_0,\xi_1\rangle$. The action on $x,y$ is by
$$
x \mapsto (\alpha x +\beta)/(\gamma x + \delta), \quad
y \mapsto y/(\gamma x + \delta)^3 \, .
$$

Let $Y$ be the algebraic stack of triples $(\pi,\alpha,\beta)$ with $\pi: C \to S$ a curve of genus $2$, $\alpha: \mathcal{O}_S^{\oplus 2} {\buildrel \sim \over \longrightarrow} 
\pi_* \omega_{\pi}$ and $\beta \in \pi_*(\omega^{\otimes 3})$ nowhere in the image
of ${\rm Sym}^3(\pi_*(\omega_{\pi}))\to \pi_*(\omega^{\otimes 3}_{\pi})$.
We may view it as the stack of curves of genus $2$ with a framed Hodge bundle and
a section of $\omega_{\pi}^{\otimes 3}$ satisfying the additional condition that it yields
an equation as  (5).
There is an obvious action of ${\rm GL}(2)$ and an action of 
${\rm Sym}^3(\mathcal{O}_S^{\oplus 2})$.

We thus consider the stack
$$
[Y/{\rm GL}(2)\ltimes {\rm Sym}^3(\mathcal{O}_S^{\oplus 2})]\, .
$$
and we can identify it with the moduli space $\mathcal{M}_2\otimes{\FF}_2$
of curves of genus $2$ in characteristic $2$.

\bigskip
We now describe a concrete form of this stack. Let $V$ be a $2$-dimensional
$k$-vector space generated by $x_1,x_2$.
Consider the subspace of  ${\rm Sym}^3(V) \times {\rm Sym}^6(V)$
of pairs $(a,b)$ satisfying the condition that the pair defines a
non-singular curve; in the affine version this amounts to
$a$ and $(a')^2b+(b')^2$ (with $a'$ and $b'$ the derivative) 
having no root in common. 
\smallskip

We  write $V_{j,l}$ for ${\rm Sym}^j(V) \otimes \det(V)^{\otimes l}$.
We let the semi-direct product 
${\rm GL}(V) \ltimes V_{3,-1}$ act on $ V_{3,-1}\times V_{6,-2}$ via twisting
the two actions of ${\rm GL}(V)$ and ${\rm Sym}^3(V)$. Without twisting an
element
$M \in {\rm GL}(V)$ acts by
$$
(a,b) \mapsto (a(\alpha x_1+\beta x_2,\gamma x_1+\delta x_2), b(\alpha x_1+\beta x_2,\gamma x_1+\delta x_2))
$$
and $v \in {\rm Sym}^3(V)$ by
$$
(a,b) \mapsto (a,b+v^2+va) \, .
$$
After twisting  $c\cdot {\rm Id}_V$ acts via $c$ on $V_{3,-1}$
and by $c^2$ on $V_{6,-2}$ and this is compatible with the equation $y^2+ay=b$
if we let $c \cdot {\rm Id}_V$ 
 act on $y$ by $y\mapsto c\, y$. It is also compatible with
$b \mapsto b+v^2+av$ for $v \in V_{3,-1}$.

We thus consider
$$
{\mathcal X}^0 \subset {\mathcal X}=
V_{3,-1}\times V_{6,-2}\, ,
$$
where ${\mathcal X}^0$ is the open substack 
given by the condition that $y^2+ay=b$ defines a 
smooth projective curve
of genus $2$. 

The moduli space ${\mathcal M}_2 \otimes {\FF}_2$ can be identified with the stack quotient
$$
[\mathcal{X}^0/{\rm GL}(V) \ltimes V_{3,-1}]
$$
Note that by our choice of twisting 
the stabilizer of a pair $(a,b)$ contains $({\rm id}_V,a)$ as it should,
since the generic curve has an automorphism group of order $2$.
The pullback of the Hodge bundle ${\EE}$ on $\mathcal{M}_2\otimes{\FF}_2$ is the 
equivariant bundle~$V$.
\end{section}
%%%%%%%%%%%%%%%%%%%%%%%%
\begin{section}{Invariants and modular forms} \label{InvariantsandMF}
We write elements $a \in {\rm Sym}^3(V)$ and  $b\in {\rm Sym}^6(V)$ as
$$
a=\sum_{i=0}^3 a_i \, x_1^{3-i}x_2^i, 
\quad b=\sum_{i=0}^6 b_i\,  x_1^{6-i}x_2^i\, .
$$
Following standard usage, by 
an invariant for the action of ${\rm GL}(V) \ltimes {\rm Sym}^3(V)$ we mean 
a polynomial in the coordinates $a_0,\ldots,a_3$ and $b_0,\ldots,b_6$ of the pair $(a,b)$
in ${\rm Sym}^3(V) \times {\rm Sym}^6(V)$
that is invariant under ${\rm SL}(V) \ltimes {\rm Sym}^3(V)$.
We define $\mathcal{K}$ as the ring of ${\rm SL}(V)\ltimes {\rm Sym}^3(V)$-invariants 
under the action on $\mathcal{X}$.

The simplest example of such an invariant is the expression
$$
K_1= a_0a_3+a_1a_2,
$$
the square root of the discriminant of $a$.

Since the pullback to $\mathcal{X}^0$ 
of the Hodge bundle ${\EE}$ under the Torelli map
$\mathcal{M}_2 \to \mathcal{A}_2$ is the equivariant bundle $V$, 
pulling back of scalar-valued modular forms gives invariants. 
This provides us with a homomorphism
$$
\mu: \mathcal{R}_2({\FF}_2) \to \mathcal{K}\, .
$$
Since the image of the Torelli map 
$\mathcal{M}_2  \hookrightarrow \mathcal{A}_2$ is 
the open set that is the complement of $\mathcal{A}_{1,1}$, 
the zero locus of $\chi_{10}$,
we see that an invariant defines a  rational modular form that becomes
regular after multiplication by a power of $\chi_{10}$.
Therefore, we can extend the map $\mu$ to
$$
 \mathcal{R}_2({\FF}_2) {\buildrel \mu \over \longrightarrow} \mathcal{K}
{\buildrel \nu \over \longrightarrow} \mathcal{R}_2({\FF}_2)_{\chi_{10}} 
$$
such that $\nu\cdot \mu={\rm id}$; here the index $\chi_{10}$ indicates 
localization at the multiplicative system generated by $\chi_{10}$.

Since we know the existence of a modular form $\psi_1 \in M_1(\Gamma_2)$
we see that $\mu(\psi_1)$ must equal $K_1$ as $K_1$ is the only invariant
of weighted degree $3$. (Here the weight of $a_i$ (resp.\ $b_j$) is $i$ (resp.\ $j$).)
This implies in particular that $\nu(K_1)$ is a 
regular modular form of weight $1$.

\bigskip

Similarly for covariants. 
By a covariant under the action of ${\rm GL}(V) \ltimes {\rm Sym}^3(V)$ we mean
a polynomial in the coefficients $a_i$ and $b_j$ and $x_1$, $x_2$ which is
an invariant for the action of ${\rm SL}(V)\ltimes {\rm Sym}^3(V)$.

The simplest example of a covariant is given by the polynomial $a$.
It defines an a priori rational modular form of weight $(3,-1)$.

\begin{lemma}
The form $\nu(a)$ is  a regular modular form of weight $(3,-1)$.
\end{lemma}
\begin{proof}
The pullback of the ${\rm Sym}^3({\EE})$ to 
 $(\mathcal{A}_1 \times \mathcal{A}_1) \otimes {\FF}_2$
decomposes as 
$$
\bigoplus_{j=0}^3\,  p_1^*({\EE}_1)^{3-j}\otimes p_2^*({\EE}_1)^j
$$ 
and the coefficients of $a$ correspond to the factors. In view of the symmetry
we have
${\rm ord}_{\mathcal{A}_{1,1}}(a_0)={\rm ord}_{\mathcal{A}_{1,1}}(a_3)$ and
${\rm ord}_{\mathcal{A}_{1,1}}(a_1)={\rm ord}_{\mathcal{A}_{1,1}}(a_2)$. 
As we saw above the form
$\nu(K_1)=\nu(a_0a_3+a_1a_2)$ is regular. By coordinate changes
 we may assume that $a_0=0$ or $a_1=0$, so it easily follows that 
${\rm ord}_{\mathcal{A}_{1,1}}(a_i)\geq 0$ and the form is regular.
Moreover, since its pullback lands in $\oplus_{j=0}^3 M_{2-j}(\Gamma_1) \otimes
M_{j-1}(\Gamma_1)$ and $M_{2}(\Gamma_1)\otimes M_{-1}(\Gamma_1)=(0)$ it follows that
${\rm ord}_{\mathcal{A}_{1,1}}(a_0) \geq 1$ 
and ${\rm ord}_{\mathcal{A}_{1,1}}(a_3)\geq~1$. 
\end{proof}

\begin{corollary} \label{ordai}
We have ${\rm ord}_{\mathcal{A}_{1,1}}(a_i)=0$ for $i=1,2$ and ${\rm ord}_{\mathcal{A}_{1,1}}(a_i)\geq 1$ for $i=0,3$.
\end{corollary}
\begin{proof}
We note that $\psi_1$ does not vanish on $\mathcal{A}_{1,1}\otimes {\FF}_2$.
\end{proof}
\begin{remark}
The restriction of the modular form 
$\nu(a)$ to $\mathcal{A}_{1,1}\otimes {\FF}_2$ is the form
$c_1\otimes 1 + 1 \otimes c_1 \in M_1(\Gamma_1)\otimes M_0(\Gamma_1)\oplus M_0(\Gamma_1)
\otimes M_1(\Gamma_1)$.
\end{remark}
Another example is the expression
$$
(a_0a_2+a_1^2)x_1^2+(a_0a_3+a_1a_2)x_1x_2+(a_1a_3+a_2^2)x_2^2
$$
which then defines a regular modular form of weight $(2,0)$. In fact,
this lives in the second factor of
the decomposition ${\rm Sym}^2(V_{3,-1})= V_{6,-2}\oplus V_{2,0}$.
Note that the map $V_{2,0}\to V_{0,0}$ given by
$\alpha x_1^2+\beta x_1x_2+\gamma x_2^2 \mapsto \beta$ defines an invariant
in characteristic $2$.
\end{section}
%%%%%%%%%%%%%%%%%%%%%%
\begin{section}{Modular Forms from Invariants}
Igusa used in \cite{Igusa1960}
invariants of binary sextics to construct modular forms and to
construct a coarse moduli space of curves of genus $2$.
The ${\ZZ}$-algebra of even degree invariants of binary sextics
describes a coarse moduli space 
$$
\mathcal{Y}={\rm Proj}\left({\ZZ}[J_2,J_4,J_6,J_8,J_{10}]/(J_4^2-J_2J_6+4\, J_8)\right)
$$
for curves of genus $2$ over ${\ZZ}$. 

To get characteristic $2$ invariants (or covariants) 
one lifts the curve given by $y^2+ay=b$ 
to the Witt ring, say defined by $(\tilde{a},\tilde{b})$,
 and takes an invariant (or covariant) 
of the binary sextic defined by $\tilde{a}^2+4\tilde{b}$,
 divides these by an appropriate power
of $2$ and reduces these modulo $2$. 
This defines invariants (or covariants). 

If $f=\sum_{i=0}^6 c_i x^{6-i}$ is the (universal) binary sextic then
the invariant
$$
J_2= 2^{-2}(-120 \, c_0c_6+20\, c_1c_5-8\, c_2c_4+3 \, c_3^2)
$$
gives by the substitution 
$$
c_i= 4\, b_i+ \sum_{0\leq r,s\leq 3, r+s=i} a_ra_s \eqno(6)
$$
the invariant
$ K_2=(a_0a_3+a_1a_2)^2$, that is, $K_2=K_1^2$. For the formula for $J_2$ 
and the formulas of the
invariants $J_{2i}$ ($i=2,3,5$) 
used below we refer to \cite[p.\ 139--140]{QL} and \cite[p.\ 204]{QL-MA}.

Similarly, the invariant $J_4$ given by
$$
2^{-7} (2640 \, c_0^2c_6^2 -880 \, c_0c_1c_5c_6+ \ldots)
$$
yields via $I_4=J_2^2-24\, J_4$ an invariant $K_4$ in characteristic $2$ that turns out to be
reducible and divisible by $K_1$. We thus set $K_3=K_4/K_1$ and get an invariant
$$
\begin{aligned}
K_3=(a_0a_3+a_1a_2 ) b_3^2+ 
(a_0^2a_3^2+a_0a_2^3+a_1^3a_3+a_1^2a_2^2 )b_3+&\\ 
(a_0a_3+a_1a_2)a_1^2 b_4+  ( a_0a_3+a_1a_2 ) a_2^2b_2+ &\\
( a_0^2a_1a_3+a_0^2a_2^2+a_0a_1^2a_2+a_1^4) b_5+ &\\
( a_0a_2a_3^2+a_3^2a_1^2+a_1a_2^2a_3+a_2^4 ) b_1+ &\\
( a_0a_3+a_1a_2 ) a_0^2b_6+ 
( a_0a_3+a_1a_2 ) a_3^2b_0\, .  &\\
\end{aligned}
$$
Similarly, $J_6$ yields the invariant $K_3^2$ in characteristic $2$.
The invariant $J_8$ yields the invariant
$$
\begin{aligned}
K_8=  b_3^8+ & (a_0a_3+a_1a_2)^2 b_3^6 + \\
& (a_0a_3+a_1a_2) (a_0^2a_3^2+a_0 a_2^3+ a_1^3a_3+a_1^2a_2^2) b_3^5 +\ldots\, , \\ 
\end{aligned}
$$
and similarly from $J_{10}$ we obtain
$$
\begin{aligned}
K_{10}= b_3^6 (a_0a_3)^4  + b_3^5 (a_0^5a_3^5+a_0^4a_1a_2a_3^4+a_0^4a_2^3a_3^3+a_0^3a_1^3a_3^4) + & \\
b_3^4 \big(a_0^6a_3^6+a_0^5a_1a_2a_3^5+a_0^4a_1^2a_2^2a_3^4+a_0^3a_1^3a_2^3a_3^3+&\\
a_0^6a_3^4b_6+a_0^5a_1a_3^4b_5+a_0^5a_2^2a_3^3b_5+a_0^4a_1^2a_3^4b_4+&\\
a_0^4a_1a_2^3a_3^2b_5+a_0^4a_2^6b_6+a_0^4a_2^5a_3b_5+a_0^4a_2^4a_3^2b_4+&\\
a_0^4a_2^2a_3^4b_2+a_0^4a_2a_3^5b_1+a_0^4a_3^6b_0+a_0^3a_1^2a_3^5b_1+&\\
a_0^2a_1^4a_3^4b_2+a_0^2a_1^3a_2a_3^4b_1+a_0a_1^5a_3^4b_1+a_1^6a_3^4b_0+&\\
a_0^4a_2^4b_5^2+a_1^4a_3^4b_1^2\big) +\ldots \, .  &\\
\end{aligned}
$$

Finally, the invariant $I_{15}$ defines an invariant $K_{15}$, but 
$K_{15}$  can be expressed in terms of the invariants already obtained:
$$ 
K_{15}=K_1^3 K_3^4+K_1^5 K_{10}+K_1^4 K_3K_8+K_3^5.
$$
\begin{lemma}
We have
${\rm ord}_{\mathcal{A}_{1,1}}(b_3)=-1$ and
${\rm ord}_{\mathcal{A}_{1,1}}(b_i)\geq 0$ if $i\neq 3$.
\end{lemma}
\begin{proof} 
In characteristic $0$ we have $\dim S_{6,8}(\Gamma_2)=1$. 
The covariant $\tilde{f}$, the universal binary sextic, 
defines a rational modular form $\nu(\tilde{f})$
such that $\chi_{10}\nu(\tilde{f}) \in S_{6,8}(\Gamma_2)$.
A non-zero element $\chi$ of $S_{6,8}(\Gamma_2)$ vanishes on $\mathcal{A}_{1,1}$ and the 
coordinates of $\chi/\chi_{10}$ thus have order $\geq -1$ along $\mathcal{A}_{1,1}$.
Since the $a_i$ are regular near
$\mathcal{A}_{1,1}\otimes {\FF}_2$ we see using the formulas (6) 
that ${\rm ord}_{\mathcal{A}_{1,1}}(b_i) 
\geq -1$.
By an argument as in Lemma \ref{dim68} if the vanishing order
equals $-1$, the only non-zero term in the first (`non-constant') term
of the Taylor development sits in the middle coordinate and lands in
$S_{12}(\Gamma_1)\otimes S_{12}(\Gamma_1)$ and thus we see 
that ${\rm ord}_{\mathcal{A}_{1,1}}(b_i)\geq 0$ for
$i\neq 3$. 
The invariant $K_3$ defines a rational modular form $\nu(K_3)$ 
of weight $3$. Since it is not a multiple of $\psi_1^3$ and $\dim M_{3}(\Gamma_2)=1$
it cannot be regular.
Therefore ${\rm ord}(b_3)=-1$. 
\end{proof}

Using this lemma we obtain modular forms 
$\psi_1=\nu(K_1)\in M_1(\Gamma_2)$
and
$$
\chi_{13}=\nu(K_3) \chi_{10} \in S_{13}(\Gamma_2)
$$
as the form of $K_3$ shows that $\nu(K_3)$ has order $-2$ along 
$\mathcal{A}_{1,1}$ since $K_1$ does not vanish on $\mathcal{A}_{1,1}$.
We see that $\chi_{13}$ is not divisible by $\psi_1$, hence 
$\dim M_{12}(\Gamma_2)
< \dim M_{13}(\Gamma_2).$
In view of the  estimates on the dimensions of $M_k(\Gamma_2)$ 
implied by (2)  we conclude
that 
$$
\dim M_{12}(\Gamma_2)=3 \quad \text{\rm and} \quad  \dim M_{13}(\Gamma_2)=4 \, .
\eqno(7)
$$

By inspecting the expression of the invariant 
$K_{10}$ in terms of the $a_i$ and $b_j$ 
we see that $\nu(K_{10})$ is a modular form of weight $10$ 
that vanishes
with multiplicity $2$ along $\mathcal{A}_{1,1}$. Therefore, by the cycle
class of $\mathcal{A}_{1,1}$ 
given in (1) it is a multiple of $\chi_{10}$ and we normalize $\chi_{10}$  by setting
$$
\chi_{10}=\nu(K_{10})\, .
$$

Consider $K_{12}=K_8 K_1^4+K_3^4+K_1^3K_3^3$. It starts as follows
$$
(a_0^2a_3^2+a_0a_1a_2a_3+a_0a_2^3+a_1^3a_3)^4 \, b_3^4 + \ldots
$$
and one can check that it has order $0$ along $\mathcal{A}_{1,1}$. 
This defines a modular form
$$
\psi_{12}=\nu(K_{12})\, .
$$
\begin{remark}
From the fact that $\chi_{10}$ (resp.\ $\psi_{12}$) 
vanishes with multiplicity
$2$ (resp.\ $0$) along $\mathcal{A}_{1,1}$ we see that the order of $a_0$ and $a_3$ along $\mathcal{A}_{1,1}$ is $1$.
\end{remark}
Similarly, $\nu(K_8)$ has order $-8$ along $\mathcal{A}_{1,1}$,
hence we find
$$
\chi_{48}=\nu(K_8) \chi_{10}^4 \in S_{48}(\Gamma_2)\, .
$$
In this way we obtain modular forms $\psi_1,\chi_{10},\psi_{12}, \chi_{13}, \chi_{48}$.

The relation $K_{12}=K_8K_1^4+K_3^4+K_1^3K_3^3$ implies  that
$$
\chi_{10}^4 \psi_{12}= \chi_{48} \psi_1^4 + \chi_{13}^4+ \chi_{13}^3 \psi_1^3
\chi_{10}\, .
\eqno(8)
$$

\bigskip

Inside the ring $\mathcal{R}_2({\FF}_2)$ we thus found generators
$$
\psi_1,\chi_{10},\psi_{12}, \chi_{13}, \chi_{48}
$$
with a relation (8) of weight $52$.
These generate a subring $R=\oplus_k R_k$ of $\mathcal{R}_2({\FF}_2)$ 
with generating function for the
dimensions
$$
G=
\frac{1-t^{52}}{(1-t)(1-t^{10})(1-t^{12})(1-t^{13})(1-t^{48})}
$$
and $\dim R_k= k^3/1080+O(k^2)$.
To see that $R=\mathcal{R}_2({\FF}_2)$ one can use  the degree
of ${\rm Proj}(R)$ 
$$
2\, \deg(\lambda_1^3)=\frac{1}{1440}= 
\frac{52}{1\cdot 10\cdot 12 \cdot 13\cdot 48}  \, ,
$$
(see \cite{vdG})
or argue with the dimensions as follows.
Let $d(k)=\dim M_{k}(\Gamma_2)$ and $r(k)=\dim R_k$. 

\begin{proposition} We have $d(k)=r(k)$ for $k\geq 0$.
\end{proposition}
\begin{proof}
We have $d(k)\geq r(k)$
and we observed equality for $k=0,\ldots,13$ using the upper bound provided
by (2), see (3) and (7). 
We assume by induction that we have $d(k)=r(k)$ for $k \leq m$.
Then by (2) we get an upper bound $d(k) \leq r(k-10) + c(k)(c(k)+1)/2$ 
for $k \leq m+10$ with $c(k)=\dim M_{k}(\Gamma_1)= \lfloor k/12\rfloor+1$.
But we have $r(k)-r(k-10)=c(k)(c(k)+1)/2$ for 
$k\not\equiv 0,\, 1, \, 2 \, (\bmod \, 12)$. Indeed, 
$$
G-t^{10}G- \frac{1}{(1-t)(1-t^{12})^2} = - t^{12} \, 
\frac{t^{26}+t^{25}+t^{24}+t^{13}+t^{12}+1}{t^{60}-t^{48}-t^{12}+1} \, .
$$
Hence $d(k)=r(k)$ for $k \leq m+10$ and $k\not\equiv 0,\, 1, \, 2 \, (\bmod \, 12)$.
But using $d(k+1)-d(k)\geq r(k+1)-r(k)$, to be proved in the next lemma, we
get $d(k)=r(k)$ for all $k \leq m+10$. 
Induction finishes the proof.
\end{proof}

\begin{lemma} We have $d(k+1)-d(k)\geq r(k+1)-r(k)$ for $k\geq 0$.
\end{lemma}
\begin{proof} Because of (8) 
we have $R_{k+1}=\psi_1 R_k \oplus N_{k+1}$ with
$N_{k+1}$ the subspace with basis
the forms $\chi_{10}^a\psi_{12}^b\chi_{13}^c\chi_{48}^d$ with
$10\, a+12\, b+13\, c+48\, d=k+1$ with $a,b,c,d \geq 0$ and $c\leq 3$.
We thus have $r(k+1)-r(k)=\dim N_{k+1}$.
We claim that $N_{k+1} \cap \psi_1 M_k(\Gamma_2)=(0)$, and this implies
$d(k+1)-d(k)\geq \dim N_{k+1}$.
Suppose we have an $f\in M_{k}(\Gamma_2)$ with $f\not\in R_k$ and
$\psi_1f \in N_{k+1}$.
We can write $\psi_1f= P$ with $P$ a polynomial in
$\chi_{10},\psi_{12},\chi_{13},\chi_{48}$,
 where no power of $\chi_{13}$ is $>3$.
Then $P=\nu(Q)$ with $Q$ a polynomial in
$$
K_{10},\,  K_8K_1^4+K_3^4+K_3^3K_1^3, \,  K_{3}K_{10}, \, K_8 K_{10}^4
$$
that is divisible by $K_1$ by assumption.
In order that a non-zero $Q$ is divisible by $K_1$ we need at least one monomial
of $Q$  
involving a  power $\geq 4$ 
of $K_3K_{10}$ since $K_1, K_3, K_8$ and $K_{10}$ are algebraically independent.
But this is excluded.
\end{proof}

\begin{remark}
In characteristic different from $2$ 
the invariant $E$ of degree $15$ gives rise to a generator,
a modular form $\chi_{35}$ of weight $35$. But in our case we get
$$
K_{15}=K_3K_{12}+K_{10}K_1^5
$$
and $\nu(K_{15})$ has order $-2$ along $\mathcal{A}_{1,1}\otimes {\FF}_2$,
so if we set $\chi_{25}=\chi_{10}\, \nu(K_{15}) \in S_{25}(\Gamma_2)$ 
we have
$$
\chi_{25}= \chi_{13}\psi_{12}+ \chi_{10}^2\psi_1^5 
$$ 
and do not get a new generator.
\end{remark}

\begin{remark}
The zero locus of $\psi_1$ is the closure of the locus $V_1$ of abelian surfaces of
$2$-rank $\leq 1$. The zero locus of the ideal
$(\psi_1,\chi_{13})$ consists of two components:
the locus $V_0$ of $2$-rank $0$, that is, the supersingular locus,
and the closure of the intersection of $\mathcal{A}_{1,1}\cap V_1$, both with
multiplicity $2$. This fits the cycles classes:
$$
13 \lambda_1^2 = 3 \lambda_1^2 + 10 \lambda_1^2\, .
$$
The class of $V_0$ is $3 \lambda_2= (3/2) \lambda_1^2$, with $\lambda_2$ the second Chern class of ${\EE}$, see \cite[Thm.\ 2.4]{vdG}, 
\cite[Thm.\ 12.4]{E-vdG}.
\end{remark}
\end{section}
%%%%%%%%%%%%%%%%%%%%%%%%%%%%%%%%%%%%%%%%%%%%%%%%%%%%%%
%%%%%%%%%%%%%%%%%%%%%%%%%%%%%%%%%%%%%%%%
%%%%%%%%%%%%%%%%%%%%%%%%%%%%%%%%%%%%%%%%%%%%%%%%%%%%%%%%%%%%%%%%%%%%

\end{document}